\documentclass[10pt]{article}

\usepackage{amsthm,amsfonts, amsbsy, amssymb,amsmath,graphicx}
\usepackage{graphics}
\usepackage[english]{babel}
\usepackage{graphicx}
\usepackage{lineno}
\usepackage{enumerate}
\usepackage{tikz}
\usepackage{tkz-graph}
\usepackage{tkz-berge}
\usepackage{hyperref}
\usepackage{color}
\usepackage{tikz}

\hypersetup{colorlinks=true, linkcolor=blue, citecolor=red,urlcolor=blue}

\newtheorem{theorem}{Theorem}
\newtheorem{lemma}[theorem]{Lemma}

\newtheorem{corollary}[theorem]{Corollary}

\newtheorem{remark}[theorem]{Remark}

\title{A note on total co-independent domination in trees}

\author{Abel Cabrera Mart\'inez$^{(1)}$, Frank A. Hern\'andez Mira $^{(2)}$\\ Jos\'e M. Sigarreta Almira$^{(2)}$ and Ismael G. Yero$^{(3)}$\\
\\
$^{(1)}${\small Universitat Rovira i Virgili,}
{\small Departament d'Enginyeria Inform\`atica i Matem\`atiques } \\  {\small Av. Pa\"{\i}sos
Catalans 26, 43007 Tarragona, Spain.} \\{\small
  abel.cabrera\@@urv.cat}\\
$^{(2)}$ {\small Universidad Aut\'onoma de Guerrero, Facultad de  Matem\'{a}ticas}\\
{\small Carlos E. Adame 5, Col. La Garita, Acapulco, Guerrero, Mexico.}\\
{\small fmira8906\@@gmail.com}, {\small josemariasigarretaalmira\@@hotmail.com}\\
$^{(3)}$ {\small Universidad de C\'adiz,}
{\small Departamento de Matem\'aticas}\\{\small Av. Ram\'on Puyol s/n, 11202 Algeciras, Spain.}\\ {\small
ismael.gonzalez\@@uca.es}
}

\date{}

\begin{document}

\maketitle

\begin{abstract}
\noindent
A set $D$ of vertices of a graph $G$ is a total dominating set if every vertex of $G$ is adjacent to at least one vertex of $D$. The total domination number of $G$ is the minimum cardinality of any total dominating set of $G$ and is denoted by $\gamma_t(G)$. The total dominating set $D$ is called a total co-independent dominating set if $V(G)\setminus D$ is an independent set and has at least one vertex. The minimum cardinality of any total co-independent dominating set is denoted by $\gamma_{t,coi}(G)$. In this paper, we show that, for any tree $T$ of order $n$ and diameter at least three, $n-\beta(T)\leq \gamma_{t,coi}(T)\leq n-|L(T)|$ where $\beta(T)$ is the maximum cardinality of any independent set and $L(T)$ is the set of leaves of $T$. We also characterize the families of trees attaining the extremal bounds above and show that the differences between the value of $\gamma_{t,coi}(T)$ and these bounds can be arbitrarily large for some classes of trees.
\end{abstract}



\section{Introduction} \label{Intro}

Throughout this work we consider $G=(V(G),E(G))$ as a simple graph of order $n=|V(G)|$. That is, graphs that are finite, undirected, and without loops or multiple edges. Given a vertex $v$ of $G$, $N_G(v)$ represents the \emph{open neighborhood} of $v$, \emph{i.e.}, the set of all neighbors of $v$ in $G$ and the \emph{degree} of $v$ is $\delta(v) = |N_G(v)|$. If $S\subseteq V(G)$, then the \emph{open neighborhood} of $S$ is $N_G(S)=\cup_{v\in S} N_G(v)$. The \emph{minimum} and \emph{maximum degrees} of $G$ are denoted by $\delta(G)$ and $\Delta(G)$, respectively. For any two vertices $u$ and $v$, the \emph{distance} $d(u,v)$ between $u$ and $v$ is the length of a shortest $u-v$ path. The \emph{diameter} of $G$ is largest possible distance between any two vertices of $G$, and is denoted by $diam(G)$. For any other graph theory terminology we follow the book \cite{west}.

\noindent
A \emph{dominating set} of a graph $G$ is a set of vertices $D\subseteq V(G)$ such that every vertex of $G$ not in $D$ is adjacent to at least a vertex in $D$. The \emph{domination number} of $G$ is the minimum cardinality of a dominating set of $G$ and is denoted by $\gamma(G)$. Domination in graphs is a classical topic, and nowadays one of the most active areas, of research in graph theory. This fact can be seen for instance through the more than 1620 articles published in the topic (more than 1070 of them in the last 10 years), according to the MathSciNet database with the queries: ``domination number'' or ``dominating sets''. The two books \cite{book-dom-1,book-dom-2}, although a little not updated by now, contain a significant amount of the most important results in the topic before this new century.

\noindent
One interesting research activity on domination in graphs concerns its relationship with other graph's parameters. A remarkable case regards vertex independence, \emph{i.e.}, sets of vertices inducing an edgeless graph. The most natural relationship in this direction is clearly the mixing of both concepts which gives raise to the independent domination number, \emph{i.e.}, the minimum cardinality of a dominating set inducing an edgeless graph. Independent domination in graphs was formally introduced in \cite{berge} and \cite{ore} at earliest 1960's (a fairly complete survey on this topic was recently published in \cite{indep-dom-surv}). On the other hand, some other investigations connecting domination and independence in graphs can easily be found in the literature. We remark here two of them. The first one remarkable case is the searching of two disjoint sets in a graph, in which one of the sets is a maximal independent set and the second one a minimal dominating set (a particular case appears whether both sets form a partition of the vertex set of the graph). Several studies on this topic have been developed in the last years. The Ph. D. thesis \cite{Lowenstein} contains several results and citations on the subject. A second case is related to finding the minimum cardinality of a dominating set which intersects every maximal independent set of a graph: independent transversal domination (see \cite{hamid}, where the theme was introduced, and also \cite{yero}.)

\noindent
Other popular studies on domination in graphs deal with modifying the domination property. One of the most popular parameters in this way is the total domination, which differs from the standard domination by the added property of dominating all the vertices of the graph, instead of only that outside of the set. It is then not surprising that total domination has been also related to independence in graphs. For instance, the total domination version of independent transversal domination is already known from \cite{yero2, ITTD-tree}. However, not much exists concerning finding two disjoint sets in which one of the sets is a total dominating set, and the other one, an independent set. A closely related idea to this one above was recently published in \cite{Soner2012}, where the concept of total co-independent dominating set was introduced. Such work contains just a few and not deep results. However, the idea of the concept is interesting and deserves a more detailed study.

\noindent
More formally, a set $D\subseteq V(G)$ is a \emph{total dominating set} of $G$ if every vertex in $V(G)$ is adjacent to at least one vertex in $D$.  The \emph{total domination number} of $G$ is the minimum cardinality of any total dominating set of $G$ and is denoted by $\gamma_t(G)$. A $\gamma_t(G)$-\emph{set} is a total dominating set of cardinality $\gamma_t(G)$. For more information on total domination see the survey \cite{Henning2009}. On the other hand, a set $S$ of vertices is \emph{independent} if $S$ induces an empty graph. An independent set of maximum cardinality is a \emph{maximum independent set} of $G$. The \emph{independence number} of $G$ is the cardinality of a maximum independent set of $G$ and is denoted by $\beta(G)$. An independent set of cardinality $\beta(G)$ is called a $\beta(G)$-\emph{set}.

\noindent
A total dominating set $D$ of a graph $G$ is called a \emph{total co-independent dominating set} if the set of vertices of the subgraph induced by $V(G)\setminus D$ is independent and not empty. The minimum cardinality of any total co-independent dominating set is denoted by $\gamma_{t,coi}(G)$. A total co-independent dominating set of cardinality $\gamma_{t,coi}(G)$ is a $\gamma_{t,coi}(G)$-\emph{set}. These concepts were previously introduced and barely studied in \cite{Soner2012}, and also recently and deeper, in \cite{Cabrera2017-1}. A slightly different version of this parameter was introduced in \cite{k2}, and also recently studied in \cite{MPSY}, where the condition of $V(G)\setminus D$ being not empty was not required. In this latter works, the total co-independent domination number is called total outer-independent domination number. Even so, both parameters behaves almost always in the same manner, we prefer to continue using the terminology of \cite{Soner2012}. Clearly, the results on this parameter would lead to deduce some conclusions into the existence of partitions of the vertex set of a graph into a total dominating set and an independent set, as already mentioned.

\noindent
Let $T$ be a tree. A \emph{leaf} of $T$ is a vertex of degree one. A \emph{support vertex} of $T$ is a vertex adjacent to a leaf which is not a leaf, and a \emph{semi-support vertex} is a vertex adjacent to a support vertex that is not a leaf. By an \emph{isolated support vertex} of $T$ we mean an isolated vertex of the subgraph induced by the support vertices of $T$. The set of leaves is denoted by $L(T)$, the set of support vertices is denoted by $S(T)$, and the set of semi-support vertices is denoted by $SS(T)$. Moreover, $S^\ast(T)$ is the set of isolated support vertices of $T$. Also, a \emph{double star} is a tree with exactly two adjacent vertices of degree larger than or equal to two and the remaining vertices are leaves.

\noindent
Studies on characterizing domination related parameters in trees have been very popular in the last decade. One can find in the literature several works showing all the trees satisfying diverse properties. For instance, to just name a few of them, we remark some examples, which could probably be not all the most remarkable and/or recent cases.
\begin{itemize}
\item In \cite{Chellali2006} was proved that for any tree $T$ of order $n$ and $l$ leaves, $\gamma_t(T)\ge (n-l+2)/2$, and all the trees achieving such bound are given.
\item In \cite{shan}, a characterization of the family of trees with equal total domination and paired-domination numbers was given.
\item In \cite{harary}, a characterization of trees with equal domination and independent domination numbers was presented.
\item In \cite{domke} was proved that the restrained domination number of a tree $T$ of order $n$ is bounded below by $\lceil(n+2)/3\rceil$ and all the extremal trees achieving this lower bound were constructively characterized.
\item In \cite{alvarado}, a constructive characterization of the trees for which the Roman domination
number strongly equals the weak Roman domination number was given.
\item In \cite{hen-rall} were characterized the trees with equal total domination and game total domination number.
\item In \cite{Cabrera2018}, a constructive characterization of vertex cover Roman trees was given, that is, trees whose outer independent Roman domination number equals twice its vertex cover number.
\item In \cite{Cabrera2020}, three different characterizations concerning weak Roman domination in trees were presented.
\end{itemize}
Other styles of characterizations of domination parameters in trees were presented in \cite{dorfling,hattingh}. In this sense, in the present work we pretend to improve the visibility of this new parameter, namely the total co-independent domination number, throughout characterizing several families of trees achieving some specific values of this mentioned parameter.

\noindent
The total co-independent domination number of a graph $G$ has been introduced in \cite{Soner2012}, where a few of its combinatorial properties were dealt with. Among them, a couple of almost trivial bounds in terms of $\beta(G)$ and the order of $G$ were proved for $\gamma_{t,coi}(G)$. As example, for any graph $G$ of order $n$, $n-\beta(G)\leq \gamma_{t,coi}(G)\leq n-1$. It is readily seen that such bounds could be improved in a number of situations. For instance, in the case of a tree $T$, this can be done for the upper bound as we next show.

\begin{theorem}\label{teo1}
For any tree $T$ of order $n$  and $diam(T)\geq 3$, $n-\beta(T)\leq \gamma_{t,coi}(T)\leq n-|L(T)|.$
\end{theorem}

\begin{proof}
The lower bound was already given in \cite{Soner2012}. Moreover, the upper bound follows by the fact that the set $V(T)\setminus L(T)$ is a total co-independent dominating set, since $T$ has diameter at least three.
\end{proof}

\noindent
One can immediately think into characterizing the trees achieving the bounds for the total co-independent domination number given above. This is made in the next section. In concordance with this, now on we assume that $|S(T)|\geq 2$ since the case $diam(T)=1$ ($T$ is a $P_2$ and $\gamma_{t,coi}(T)$ is not defined) and  $diam(T)=2$ ($T$ is a star graph $S_n$ and $\gamma_{t,coi}(T)=2$) are straightforward.

\noindent
In order to easily proceed with our exposition from now on we say that a tree $T$ belongs to the family $\mathcal{T}_\beta$, if $\gamma_{t,coi}(T)= n-\beta(T)$ or $T$ is in the family $\mathcal{T}_L$, if $\gamma_{t,coi}(T)= n-|L(T)|$. Note that $\mathcal{T}_L = \mathcal{T}_\beta$ if and only if $L(T)$ is a $\beta(T)$-set.

\section{The characterizations} \label{results}

\noindent
In order to provide a constructive characterization of the trees belonging to the family $\mathcal{T}_\beta$ we need to introduce some operations to be made over a tree. In this concern, by attaching a path $P$ to a vertex $v$ of $T$ we mean adding the path $P$ and joining $v$ to a vertex of $P$.

\begin{description}
  \item[Operation $O_1$:]  Attach a path $P_1$ to a vertex of $T$, which is in some $\gamma_{t,coi}(T)$-set.
  \item[Operation $O_2$:]  Attach a path $P_2$ to a vertex of $T$, which is in some $\gamma_{t,coi}(T)$-set.
  \item[Operation $O_3$:]  Attach a path $P_4$ to a vertex $v$ of $T$, which is in some $\gamma_{t,coi}(T)$-set, by joining $v$ to a leaf of $P_4$.
  \item[Operation $O_4$:]  Attach a path $P_4$ to a vertex $v$ of $T$, which is in some $\beta(T)$-set, by joining $v$ to a support of $P_4$.
\end{description}

\noindent
Let $\mathcal{T}$ be the family of trees defined as $\mathcal{T} = \{T \mid T $ is $P_4$ or is a tree obtained from $P_4$ by a finite
sequence of operations $O_1, O_2, O_3, O_4\}$. We first show  that every tree in the family $\mathcal{T}$ belongs to the family $\mathcal{T}_\beta$.

\begin{lemma}\label{lem-right}
If $T \in \mathcal{T}$, then $T\in \mathcal{T}_\beta$.
\end{lemma}

\begin{proof}
\noindent
We proceed by induction on the number $r(T)$ of operations required to construct the tree $T$. If $r(T)=0$, then $T=P_4$ and $T \in \mathcal{T}_\beta$.  This establishes the base case. Hence, we now  assume that $k\geq 1$ is an integer and that each tree $T' \in \mathcal{T}$ with $ r(T')<k$ satisfies that $T'\in \mathcal{T}_\beta$.
Let $T \in \mathcal{T}$ be a tree with $r(T)=k$. Then, $T$ can be obtained from a tree $T' \in \mathcal{T}$ with $ r(T')=k-1$ by one of the operations $O_1, O_2, O_3$ or $O_4$. We shall prove that $T \in \mathcal{T}_\beta$. To this end, as $T'\in \mathcal{T_\beta}$, let $D'$ be a $\gamma_{t,coi}(T')$-set containing no leaves and let $B'=V(T')\setminus D'$ be a $\beta(T')$-set containing all leaves. We consider the following situations.\\

\noindent
\textbf{Case 1.} $T$ is obtained from $T'$ by operation $O_1$. Assume $T$ is obtained from $T'$ by adding the vertex $u$ and the edge $uv$, where $v\in D'$. Notice that $u$ is a leaf of $T$ and that $B'\cup \{u\}$ is a $\beta(T)$-set, otherwise there would be an independent set in $T'$ with one vertex more than in $B'$, since only one vertex has been added to $T'$ (to obtain $T$) that would increase the value for $\beta(T)$ in more than one. On the other hand, $D'$ is a total co-independent dominating set of $T$. Thus, $\gamma_{t,coi}(T)\le |D'|= |V(T)|-(\beta(T')+1)=|V(T)|-\beta(T)$, and by Theorem \ref{teo1}, $\gamma_{t,coi}(T)=|V(T)|-\beta(T)$, which means $T\in T_\beta$.\\

\noindent
\textbf{Case 2.} $T$ is obtained from $T'$ by operation $O_2$. Assume $T$ is obtained from $T'$ by adding the path $u_1u_2$ and the edge $u_1v$ where $v\in D'$. Note that the vertex $u_2$ is a leaf of $T$, and that $u_1$ is its support vertex. So, $B'\cup\{u_2\}$ is a $\beta(T)$-set, otherwise there would be an independent set in $T'$ with one vertex more than in $B'$, since only one vertex from the pair $u_1,u_2$ could be added to any independent set of $T'$ to get an independent set of $T$ of larger cardinality than that of $B'$. It is now not difficult to see that $D = D' \cup \{u_1\}$ is a total co-independent dominating set of $T$. Hence, $\gamma_{t,coi}(T)\le |D|=|D'| + 1=|V(T')|-\beta(T') + 1=|V(T)|-2 -(\beta(T)-1)+1=|V(T)|-\beta(T)$. Therefore, by Theorem \ref{teo1}, $\gamma_{t,coi}(T)= |V(T)|-\beta(T)$ and $T \in \mathcal{T}_\beta$.\\

\noindent
\textbf{Case 3.} $T$ is obtained from $T'$ by operation $O_3$. Assume $T$ is obtained from $T'$ by adding the path $P_4=h_1u_1u_2h_2$ to a vertex $v\in D'$ through the edge $vh_1$. By using some similar reasons as in the case above, it is easily seen that $B=B'\cup \{h_1,h_2\}$ is a $\beta(T)$-set and that $D=D' \cup \{u_1,u_2\}$ is a total co-independent dominating set of $T$. So, $\gamma_{t,coi}(T)\le |D|=|D'|+2=|V(T')|-\beta(T')+2=(|V(T)|-4)-(\beta(T)-2)+2=|V(T)|-\beta(T)$. Therefore, by Theorem \ref{teo1}, $\gamma_{t,coi}(T)= |V(T)|-\beta(T)$ and $T \in \mathcal{T}_\beta$.\\



\noindent
\textbf{Case 4.} $T$ is obtained from $T'$ by operation $O_4$. Assume $T$ is obtained from $T'$ by adding a path $P_4=h_1u_1u_2h_2$ to a vertex $v$ of $T'$, which is in some $\beta(T')$-set, throughout the edge $vu_1$. Note that the set $B=B' \cup \{h_1,h_2\}$ is an independent set of $T$. Moreover, since there can be at most two vertices of the path $P_4$ in any $\beta(T)$-set, it must happen that $B$ is a $\beta(T)$-set, otherwise there would be an independent set in $T'$ of cardinality larger than $\beta(T')$, which is not possible. Now, it is readily seen that $D=D'\cup \{u_1,u_2\}$ is a total co-independent dominating set in $T$. Thus, $\gamma_{t,coi}(T)\le |D|=|D'|+2=|V(T')|-\beta(T')+2=(|V(T)|-4)-(\beta(T)-2)+2=|V(T)|-\beta(T)$. Therefore, again as above, by Theorem \ref{teo1}, $\gamma_{t,coi}(T)= |V(T)|-\beta(T)$ and $T \in \mathcal{T}_\beta$.
\end{proof}

\noindent
We now turn our attention to the opposite direction concerning the lemma above. In this sense, from now on we shall need the following terminology and notation in our results. Given a tree $T$ and a set $S\subsetneq V(T)$, by $T-S$ we denote a tree obtained from $T$ by removing from $T$ all the vertices in $S$ and all its incident edges (if $S=\{v\}$ for some vertex $v$, then we simply write $T-v$). For an integer $r\geq 2$, by $Q_r$ we mean a graph which is obtained from a path $P_{r+2}=v s s_1 s_2 \ldots s_r$ by attaching a path $P_1$ to any vertex of $P_{r+2}-v$. In Figure \ref{figure-1} we show the example of $Q_5$.

\begin{figure}[h]
\centering
\begin{tikzpicture}[scale=.6, transform shape]
\node [draw, shape=circle] (s) at  (0,0) {};
\node at (0.5,-0.5) {\Large $s$};
\node [draw, shape=circle] (s1) at  (0,1.5) {};
\node [draw, shape=circle] (s2) at  (0,-1.5) {};
\node at (0.5,-1.5) {\Large $v$};

\node [draw, shape=circle] (a1) at  (1.5,0) {};
\node at (1.5,-0.5) {\Large $s_1$};
\node [draw, shape=circle] (a11) at  (1.5,1.5) {};

\node [draw, shape=circle] (a2) at  (3,0) {};
\node at (3,-0.5) {\Large $s_2$};
\node [draw, shape=circle] (a21) at  (3,1.5) {};

\node [draw, shape=circle] (a3) at  (4.5,0) {};
\node at (4.5,-0.5) {\Large $s_3$};
\node [draw, shape=circle] (a31) at  (4.5,1.5) {};

\node [draw, shape=circle] (a4) at  (6,0) {};
\node at (6,-0.5) {\Large $s_4$};
\node [draw, shape=circle] (a41) at  (6,1.5) {};

\node [draw, shape=circle] (a5) at  (7.5,0) {};
\node at (7.5,-0.5) {\Large $s_5$};
\node [draw, shape=circle] (a51) at  (7.5,1.5) {};

\draw(s)--(a1)--(a2)--(a3)--(a4)--(a5);
\draw(s)--(s1);
\draw(s)--(s2);
\draw(a1)--(a11);
\draw(a2)--(a21);
\draw(a3)--(a31);
\draw(a4)--(a41);
\draw(a5)--(a51);

\end{tikzpicture}
\caption{The structure of the tree $Q_5$.}
\label{figure-1}
\end{figure}
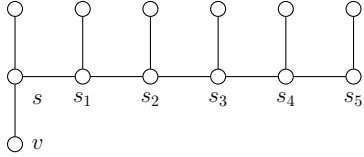

\noindent
We next show that every tree of the family $\mathcal{T}_\beta$ belongs to the family $\mathcal{T}$. To this end, we need the following remark.

\begin{remark}\label{dist-B}
Let $B$ be a $\beta(T)$-set of a tree $T$. Then for every $v\in B$ there exists a vertex $u\in B$ such that $d(v,u)\leq 3$.
\end{remark}

\begin{lemma}\label{lem-left}
If $T\in \mathcal{T}_\beta$, then $T \in \mathcal{T}$.
\end{lemma}

\begin{proof}
\noindent
We proceed by induction on the order $n\geq 4$ of the trees $T\in \mathcal{T}_\beta$. If $T$ is a double star, then $T$ can be obtained from $P_4$ by repeatedly applying operation $O_1$. This establishes the base case. We assume that $k > 4$ is an integer and that each tree $T' \in \mathcal{T}_\beta$ with $|V(T')|<k$ satisfies that $T'\in \mathcal{T}$.

\noindent
Let $T$ be a tree such that $T\in \mathcal{T}_\beta$ and $|V(T)|=k$. Let $(D,B)$ be a partition of $V(T)$ where $D$ is a $\gamma_{t,coi}(T)$-set containing no leaves and $B$ is a $\beta(T)$-set containing all leaves. We analyze the following situations.\\

\noindent
\textbf{Case 1: $|S(T)|<|L(T)|$.} We consider a support vertex $v$ that is adjacent to at least two leaves. Let $h\in N(v)\cap L(T)$ and $T'=T-h$. Since $h\in B$, it follows $B'=B\setminus \{h\}$ is an independent set in $T'$, which is moreover a $\beta(T')$-set, since otherwise we would get an independent set of larger cardinality in $T$ than $\beta(T)$, which is not possible. Moreover, it is easily seen that $D$ is also a total co-independent dominating set of $T'$. Thus, $\gamma_{t,coi}(T')\le |D|=\gamma_{t,coi}(T)=|V(T)|-\beta(T)=|V(T')|+1-(\beta(T')+1)=|V(T')|-\beta(T')$. So, by Theorem \ref{teo1} , $\gamma_{t,coi}(T')=|V(T')|-\beta(T')$ and $T' \in \mathcal{T}_\beta$. Now, by the induction hypothesis $T' \in \mathcal{T}$. Therefore, since $T$ can be obtained from $T'$ by operation $O_1$, it follows $T \in \mathcal{T}$.\\

\noindent
\textbf{Case 2: $|S(T)|=|L(T)|$ and $|SS(T)|=0$.} In this case we note that $V(T)=S(T)\cup L(T)$ and it is easily observed that $|S(T)|\geq 3$, $S(T)$ is a $\gamma_{t,coi}(T)$-set and $L(T)$ is a $\beta(T)$-set. Let $s\in S(T)$ such that $|N(s)\cap S(T)|=1$ (note that such $s$ always exists) and let $h\in L(T)$ be the leaf adjacent to $s$. First notice that the vertex $x\in N(s)\cap S(T)$ must have at least two neighbors in $S(T)$, otherwise $T$ is $P_4$, which is not possible. Hence, we observe that $B'=B\setminus\{h\}$ is an independent set of $T'=T-\{h,s\}$. Moreover, by a similar reason as in the case above, $B'$ is also a $\beta(T')$-set. On the other hand, $S(T')=S(T) \setminus \{s\}$ is clearly a total co-independent dominating set of $T'$. So, $\gamma_{t,coi}(T')\le |S(T')|=|S(T)|-1=|V(T)|-|L(T)|-1=|V(T')|+2-(|L(T')|+1)-1=|V(T')|-\beta(T')$. Thus, by Theorem \ref{teo1}, $\gamma_{t,coi}(T')=|V(T')|-\beta(T')$ and $T' \in \mathcal{T}_\beta$. By induction hypothesis $T' \in \mathcal{T}$. Since $T$ can be obtained from $T'$ by operation $O_2$, we get $T \in \mathcal{T}$.\\

\noindent
\textbf{Case 3: $|S(T)|=|L(T)|$ and $|SS(T)|>0$.}  Herein we denote by $P(x,y)$ the set of vertices of one shortest path between $x$ and $y$, including $x$ and $y$. Let $h,h'$ be two leaves at the maximum possible distance in $T$ such that there is $v \in SS(T)\cap P(h,h')$ with $d(v,h)=2$ or $d(v,h')=2$. Without loss of generality assume that $d(v,h)=2$ and let $s$ be the support adjacent to $h$. Since $|S(T)|=|L(T)|$, we observe that $N(s)\subseteq S(T)\cup \{h,v\}$ and also every support vertex is adjacent to exactly one leaf. We have now some possible scenarios.\\

\noindent
\textbf{Case 3.1  $|N(s)\cap S(T)|=1$.} Hence, by the maximality of the path $P(h,h')$, it must happen that $T$ has an induced subgraph isomorphic to a graph $Q_r$, as previously described, obtained from the vertices $v, s, h$ and some supports, say $s_1,s_2, \ldots ,s_r\in S(T)$, with the leaves $h_1,h_2,\ldots,h_r$, adjacent to the supports $s_1,s_2, \ldots ,s_r$, respectively.

\noindent
Assume $r=1$. Note that $s,s_1\in D$ and $h,h_1\in B$. We first consider $v\in D$. Let $T'=T-\{h_1,s_1\}$, $D'=D\setminus\{s_1\}$ and $B'=B\setminus\{h_1\}$. Clearly, $B'$ is a $\beta(T')$-set. Also, we note that $D'$ is a total co-independent dominating set of $T'$. So, $\gamma_{t,coi}(T')\le |D'|=|D|-1=(|V(T)|-|B|)-1=(|V(T')|+2)-(\beta(T')+1)-1=|V(T')|-\beta(T')$. Thus, by Theorem \ref{teo1}, $\gamma_{t,coi}(T')=|V(T')|-\beta(T')$ and $T' \in \mathcal{T}_\beta$. By inductive hypothesis $T' \in \mathcal{T}$, and since $T$ can be obtained from $T'$ by operation $O_2$, we obtain $T \in \mathcal{T}$.

We now consider $v\in B$. Let $T'=T-\{s,h,s_1,h_1\}$, $D'=D \setminus \{s, s_1\}$ and $B'=B\setminus\{h,h_1\}$. We can again deduce $B'$ is a $\beta(T')$-set since at most two vertices of $\{s,h,s_1,h_1\}$ can belong to any independent set of $T$. Moreover, $D'$ is a total co-independent dominating set of $T'$. So, $\gamma_{t,coi}(T')\le |D'|=|D|-2=(|V(T)|-|B|)-2=(|V(T')|+4)-(\beta(T')+2)-2=|V(T')|-\beta(T')$ and, by Theorem \ref{teo1}, we get $\gamma_{t,coi}(T')=|V(T')|-\beta(T')$, which leads to $T' \in \mathcal{T}_\beta$. By inductive hypothesis $T' \in \mathcal{T}$, and together with the fact that $T$ can be obtained from $T'$ by operation $O_4$, the required result $T \in \mathcal{T}$ follows.

\noindent
Assume now that $r\geq 2$. Since $s,s_1,\ldots,s_r\in D$ and $h,h_1,\ldots,h_r\in B$, if $T'=T-\{h_r,s_r\}$, then it is readily seen that $B'=B\setminus\{h_r\}$ is a $\beta(T')$-set and that $D'=D \setminus \{s_r\}$ is a total co-independent dominating set of $T'$. By using a similar procedure as above ($r=1$) we obtain $T' \in \mathcal{T}$ and, due to that $T$ can be obtained from $T'$ by operation $O_2$, it follows $T \in \mathcal{T}$.\\

\noindent
\textbf{Case 3.2  $|N(s)\cap S(T)|>1$.} We note that this case is analogous to that one above whether $|N(s)\cap S(T)|=1$ and $r\geq 2$.\\

\noindent
\textbf{Case 3.3: $|N(s)\cap S(T)|=0$.} Clearly, $s$ has degree two since it has one leaf neighbor, no support neighbors and cannot have more than one (it has exactly one) semi-support neighbor due to the maximality of $P(h,h')$. Also, it must happen $v,s\in D$ and $h\in B$. Assume the subgraph induced by $P(h,h')$ is $h\, s\, v\, u_1 u_2\ldots s'\, h'$, where $h, h' \in L(T)$ and $s,s'\in S(T)$. Note that $N(v)\subseteq S(T)\cup \{u_1\}$. We consider again some possible scenarios.\\

\noindent
\textbf{Case 3.3.1: $|N(v)\cap S(T)|>1$.} In this case, the vertex $v$ is also totally dominated by another support different from $s$. We note that $u_1\in B$, because $B\cup\{v\}$ would be an independent set of cardinality larger than that of $B$. Hence, we consider the tree $T'=T-\{h,s\}$ and the sets $D'=D \setminus \{s\}$, $B'=B\setminus\{h\}$. It can be deduced as above that $B'$ is a $\beta(T')$-set. Moreover, $D'$ is a total co-independent dominating set of $T'$. So, $\gamma_{t,coi}(T')\le |D'|=|D|-1=(|V(T)|-|B|)-1=(|V(T')|+2)-(\beta(T')-1)-1=|V(T')|-\beta(T')$. Thus, by Theorem \ref{teo1}, we obtain $\gamma_{t,coi}(T')=|V(T')|-\beta(T')$, which means $T' \in \mathcal{T}_\beta$, and by inductive hypothesis $T' \in \mathcal{T}$. Furthermore $T$ can be obtained from $T'$ by operation $O_2$, which allows to claim $T \in \mathcal{T}$.\\

\noindent
\textbf{Case 3.3.2: $|N(v)\cap S(T)|=1$.} Clearly $s,v$ have degree two. We firstly consider the case when $N(u_1)=\{v,u_2\}$. By Remark \ref{dist-B} we note that $u_1\in B$, and consequently $u_2\in D$. Let $T'=T-\{h,s,v,u_1\}$, $D'=D \setminus\{v,s\}$ and $B'=B\setminus\{h,u_1\}$. Now, similarly to Case 3.1 (whether $r=1$ and $v\in B$) we can deduce $B'$ is a $\beta(T')$-set. Also, $D'$ is a total co-independent dominating set of $T'$. So, $\gamma_{t,coi}(T')\le |D'|=|D|-2=(|V(T)|-|B|)-2=(|V(T')|+4)-(\beta(T')+2)-2=|V(T')|-\beta(T')$, and, by Theorem \ref{teo1}, $\gamma_{t,coi}(T')=|V(T')|-\beta(T')$. Thus, $T' \in \mathcal{T}_\beta$. Hence, the inductive hypothesis $T' \in \mathcal{T}$ together with the fact that $T$ can be obtained from $T'$ by operation $O_3$, we deduce $T\in \mathcal{T}$.

\noindent
On the other hand, assume that there is a vertex $w\in N(u_1)\setminus\{v,u_2\}$. Since $u_1\in B$ must happen, it must be also $w\in D$, and consequently $w$ has a neighbor belonging to $D$ (which is clearly not $u_1$). Moreover,  by the maximality of $P(h',h)$, we note that $w\notin SS(T)$ because $B' = B\setminus\{u_1\} \cup \{v,w\}$ is an independent set of cardinality larger than that of $B$. Hence, it follows $w\in S(T)\setminus S^\ast(T)$.

\noindent
Let $x\in V(T)$ be a leaf adjacent to the support $w$. We note that $T$ should have an induced subgraph isomorphic to a graph $Q_r$, obtained from the vertices $u_1, w, x$ and some supports, say $w_1,w_2, \ldots w_r\in S(T)$, with the leaves $x_1,x_2,\ldots,x_r$, adjacent to the supports $w_1,w_2, \ldots w_r$, respectively. By using a similar procedure to that of Case $3.1$, it follows $T \in \mathcal{T}$ which completes the proof.
\end{proof}

\noindent
As an immediate consequence of Lemma \ref{lem-right} and Lemma \ref{lem-left} we have the following characterization.

\begin{theorem}\label{teo-lower}
Let $T$ be a tree. Then $T\in \mathcal{T}_\beta$ if and only if $T \in \mathcal{T}$.
\end{theorem}

\noindent
We next see that all the operations $O_1$ to $O_4$ are required in the characterization above. First, it is easy to see that operation $O_1$ is required to obtain a double star from the path $P_4$. The examples given in Figure \ref{figure-2} show that operations $O_2$, $O_3$ and $O_4$ are also required.

\begin{figure}[ht]
\centering
\begin{tikzpicture}[scale=.5, transform shape]
\node [draw, shape=circle] (s1) at  (-1.5,0) {};
\node [draw, shape=circle] (s2) at  (0,0) {};
\node [draw, shape=circle] (s3) at  (1.5,0) {};
\node [draw, shape=circle] (s4) at  (3,0) {};

\node [draw, shape=circle] (s21) at  (0,-1.5) {};
\node [draw, shape=circle] (s22) at  (0,-3) {};
\node [draw, shape=circle] (s23) at  (0,-4.5) {};
\node [draw, shape=circle] (s24) at  (0,-6) {};

\node [draw, shape=circle] (s31) at  (1.5,-1.5) {};
\node [draw, shape=circle] (s32) at  (1.5,-3) {};
\node [draw, shape=circle] (s33) at  (1.5,-4.5) {};
\node [draw, shape=circle] (s34) at  (1.5,-6) {};

\draw(s1)--(s2)--(s3)--(s4);

\draw(s2)--(s21)--(s22)--(s23)--(s24);

\draw(s3)--(s31)--(s32)--(s33)--(s34);

\node [left] at (-1.3,-5.25) {\Large(I)};

\end{tikzpicture}
\hspace*{0.7cm}
\begin{tikzpicture}[scale=.5, transform shape]
\node [draw, shape=circle] (s1) at  (0,0) {};
\node [draw, shape=circle] (s2) at  (1.5,0) {};
\node [draw, shape=circle] (s3) at  (3,0) {};
\node [draw, shape=circle] (s4) at  (-1.5,0) {};
\node [draw, shape=circle] (s5) at  (-3,0) {};
\node [draw, shape=circle] (s6) at  (-4.5,0) {};

\node [draw, shape=circle] (s21) at  (1.5,1.5) {};
\node [draw, shape=circle] (s31) at  (3,1.5) {};

\node [draw, shape=circle] (s11) at  (0,-1.5) {};
\node [draw, shape=circle] (s12) at  (0,-3) {};
\node [draw, shape=circle] (s111) at  (-1.5,-1.5) {};
\node [draw, shape=circle] (s1111) at  (-1.5,-3) {};

\draw(s1)--(s2)--(s3);
\draw(s1)--(s4)--(s5)--(s6);

\draw(s2)--(s21);
\draw(s3)--(s31);

\draw(s1)--(s11)--(s12);

\draw(s11)--(s111)--(s1111);

\node [left] at (-2.8,-2.25) {\Large(II)};

\end{tikzpicture}
\hspace*{0.7cm}
\begin{tikzpicture}[scale=.5, transform shape]
\node [draw, shape=circle] (s1) at  (0,0) {};
\node [draw, shape=circle] (s2) at  (0,1.5) {};
\node [draw, shape=circle] (s3) at  (0,3) {};
\node [draw, shape=circle] (s4) at  (0,4.5) {};
\node [draw, shape=circle] (s5) at  (0,6) {};

\node [draw, shape=circle] (s31) at  (-1.5,3) {};

\draw(s1)--(s2)--(s3)--(s4)--(s5);
\draw(s3)--(s31);

\node [left] at (-1.3,0.75) {\Large(III)};
\end{tikzpicture}

\caption{The tree (I) can only be obtained from $P_4$ by a sequence of operations $O_3, O_3$ or $O_4, O_3$; the tree (II) can only be obtained from $P_4$ by a sequence of operations $O_4, O_4$ or $O_3,O_4$ and the tree (III) can only be obtained from $P_4$ by the operation $O_2$.}
\label{figure-2}
\end{figure}
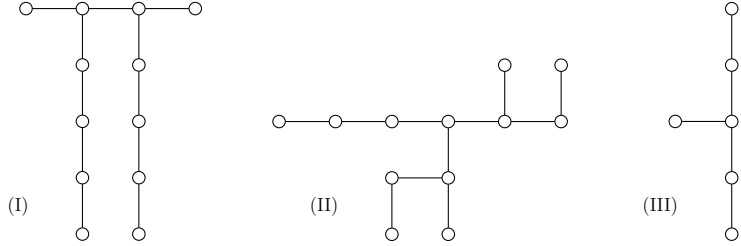

\noindent
Now, in order to characterize the trees $T$ such that $T\in \mathcal{T}_L$, we need some primary results, which we next present.

\begin{theorem}\label{teo-min}{\em \cite{Soner2012}}
A total co-independent dominating set $D$ of a graph $G$ is minimal if and only if for each vertex $v\in D$, one of the following conditions is satisfied.
\begin{itemize}
  \item[\em{(i)}] There exists a vertex $u \in V(G)$ such that $ N(u)\cap D=\{v\}$.
  \item[\em{(ii)}] There exists $w\in V(G)\setminus D$ adjacent to $v$.
\end{itemize}
\end{theorem}

\begin{lemma}\label{lem-upper}
If $T\in \mathcal{T}_L$ and $S^\ast(T)\neq \emptyset$, then $SS(T)\subseteq N(S^\ast(T))$.
\end{lemma}

\begin{proof}
\noindent
Let $D$ be a $\gamma_{t,coi}(T)$-set containing no leaves. Clearly, $T\in \mathcal{T}_L$ implies that every vertex of $T$ is either a leaf or belongs to $D$. Note that if $S^\ast(T)\neq \emptyset$, then $SS(T)\neq \emptyset$. Let $v\in SS(T)$. Since $D$ is a $\gamma_{t,coi}(T)$-set and $v$ is not adjacent to a leaf, by Theorem \ref{teo-min}, there exists a vertex $u\in D$ such that $ N(u)\cap D=\{v\}$. Hence, it is easily observed that $u\in S^\ast(T)$. Therefore, $v\in  N(S^\ast(T))$, which completes the proof.
\end{proof}

\begin{corollary}\label{cor-upper}
If $T\in \mathcal{T}_L$ and $S^\ast(T)=\emptyset$, then $V(T)=L(T)\cup S(T)$.
\end{corollary}

The next theorem contains a characterization for those trees satisfying $T\in \mathcal{T}_L$.

\begin{theorem}\label{teo-upper}
Let $T$ be any tree of order $n$ and $diam(T)\ge 3$. Then $T\in \mathcal{T}_L$ if and only if $V(T)=L(T)\cup S(T)\cup SS(T)$ and $SS(T)\subseteq N(S^\ast(T))$.
\end{theorem}

\begin{proof}
\noindent
We first assume that $T\in \mathcal{T}_L$ and, since $diam(T)\ge 3$, let $D$ be a $\gamma_{t,coi}(T)$-set containing no leaves. If $S^\ast(T)=\emptyset$, then, by Corollary \ref{cor-upper}, it follows that $V(T)=L(T)\cup S(T)$. If $S^\ast(T)\neq \emptyset$, then, by using Lemma \ref{lem-upper}, we get that $SS(T)\subseteq N(S^\ast(T))$. Hence, we note that $D\setminus(SS(T)\cup S(T))$ is empty. Otherwise, there exists $v\in N(SS(T))\setminus S(T)$ such that $D\setminus\{v\}$ is a total co-independent dominating set, which contradicts the fact that $T\in \mathcal{T}_L$. Therefore $V(T)=L(T)\cup S(T)\cup SS(T)$ and $SS(T)\subseteq N(S^\ast(T))$.

\noindent
On the other hand, if we consider that $V(T)=L(T)\cup S(T)\cup SS(T)$ and $SS(T)\subseteq N(S^\ast(T))$, then it is readily seen that $T\in \mathcal{T}_L$, which completes the proof.
\end{proof}

\noindent
An interesting question that arises from the Theorems \ref{teo-lower} and \ref{teo-upper} is the following. Can the differences $\gamma_{t,coi}(T)-(|V(T)|-\beta(T))$ and $(|V(T)|-|L(T)|)-\gamma_{t,coi}(T)$ be as large as possible? We next give an affirmative answer to that question. In this sense, the following family of trees $\mathcal{F}$ is required. Given two integers $b,d$, a tree $T_{b,d}\in \mathcal{F}$ is defined as follows.
\begin{itemize}
  \item We begin with a tree $T$ of order $n=b+d$ with vertex set $V(T)=\{u_1,\ldots,u_b,v_1,\ldots,v_d\}$.
  \item Attach two paths $P_1$ to every vertex of $T$.
  \item Attach a star $S_3$ to every vertex $u_i\in V(T)$, $i\in \{1,\ldots,b\}$, by adding an edge between $u_i$ and a leaf of the star $S_3$.
  \item Attach a star $S_3$ with a subdivided edge to every vertex $v_i\in V(T)$, $i\in \{1,\ldots,d\}$, by adding an edge between $v_i$ and the leaf corresponding to the subdivided edge.
\end{itemize}
An example of a tree of the family $\mathcal{F}$ is given in Figure \ref{family-F-fig}.
\begin{figure}[h]
\centering
\hspace*{0.7cm}
\begin{tikzpicture}[scale=.5, transform shape]
\node [draw, shape=circle] (u1) at  (1.5,-2) {};
\node [draw, shape=circle] (u2) at  (5.5,-2) {};
\node [draw, shape=circle] (v1) at  (9.5,-2) {};
\node [draw, shape=circle] (v2) at  (13.5,-2) {};
\node [draw, shape=circle] (v3) at  (17.5,-2) {};

\node [draw, shape=circle] (s1) at  (0,0) {};
\node [draw, shape=circle] (s2) at  (1.5,0) {};
\node [draw, shape=circle] (s3) at  (3,0) {};
\node [draw, shape=circle] (s4) at  (4,0) {};
\node [draw, shape=circle] (s5) at  (5.5,0) {};
\node [draw, shape=circle] (s6) at  (7,0) {};
\node [draw, shape=circle] (s7) at  (8,0) {};
\node [draw, shape=circle] (s8) at  (9.5,0) {};
\node [draw, shape=circle] (s9) at  (11,0) {};
\node [draw, shape=circle] (s10) at  (12,0) {};
\node [draw, shape=circle] (s11) at  (13.5,0) {};
\node [draw, shape=circle] (s12) at  (15,0) {};
\node [draw, shape=circle] (s13) at  (16,0) {};
\node [draw, shape=circle] (s14) at  (17.5,0) {};
\node [draw, shape=circle] (s15) at  (19,0) {};

\node [draw, shape=circle] (r1) at  (9.5,2) {};
\node [draw, shape=circle] (r2) at  (13.5,2) {};
\node [draw, shape=circle] (r3) at  (17.5,2) {};

\node [draw, shape=circle] (ss1) at  (0,4) {};
\node [draw, shape=circle] (ss2) at  (1.5,4) {};
\node [draw, shape=circle] (ss3) at  (3,4) {};
\node [draw, shape=circle] (ss4) at  (4,4) {};
\node [draw, shape=circle] (ss5) at  (5.5,4) {};
\node [draw, shape=circle] (ss6) at  (7,4) {};
\node [draw, shape=circle] (ss7) at  (8,4) {};
\node [draw, shape=circle] (ss8) at  (9.5,4) {};
\node [draw, shape=circle] (ss9) at  (11,4) {};
\node [draw, shape=circle] (ss10) at  (12,4) {};
\node [draw, shape=circle] (ss11) at  (13.5,4) {};
\node [draw, shape=circle] (ss12) at  (15,4) {};
\node [draw, shape=circle] (ss13) at  (16,4) {};
\node [draw, shape=circle] (ss14) at  (17.5,4) {};
\node [draw, shape=circle] (ss15) at  (19,4) {};

\draw(u1)--(u2)--(v1)--(v2)--(v3);

\draw(s1)--(u1)--(s3);
\draw(s4)--(u2)--(s6);
\draw(s7)--(v1)--(s9);
\draw(s10)--(v2)--(s12);
\draw(s13)--(v3)--(s15);

\draw(ss1)--(ss2)--(ss3);
\draw(ss4)--(ss5)--(ss6);
\draw(ss7)--(ss8)--(ss9);
\draw(ss10)--(ss11)--(ss12);
\draw(ss13)--(ss14)--(ss15);

\draw(u1)--(s2)--(ss2);
\draw(u2)--(s5)--(ss5);
\draw(v1)--(s8)--(r1)--(ss8);
\draw(v2)--(s11)--(r2)--(ss11);
\draw(v3)--(s14)--(r3)--(ss14);

\node [below] at (1.5,-2.5) {\Large $u_1$};
\node [below] at (5.5,-2.5) {\Large $u_2$};
\node [below] at (9.5,-2.5) {\Large $v_1$};
\node [below] at (13.5,-2.5) {\Large $v_2$};
\node [below] at (17.5,-2.5) {\Large $v_3$};

\end{tikzpicture}
\caption{The tree $T_{2,3}$ by taking $T$ as the path $P_5$.}
\label{family-F-fig}
\end{figure}
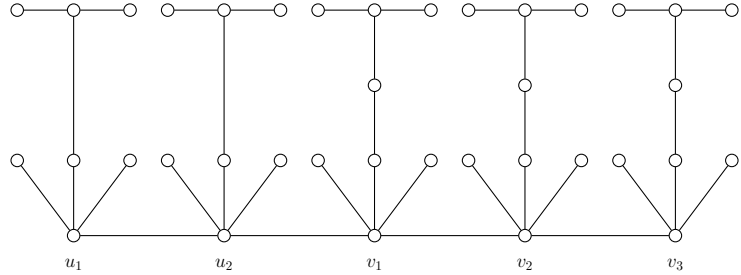
\noindent
We next give several properties of the trees of the family $\mathcal{F}$, which are almost straightforward to observe and, according to this fact, the proofs are left to the reader.

\begin{remark}\label{family-F}
Let $b,d$ be any two positive integers. Then,
\begin{itemize}
  \item[{\em (i)}] $T_{b,d}$ has order $3n+4b+5d$,
  \item[{\em (ii)}] $T_{b,d}$ has $2n+2b+2d$ leaves,
  \item[{\em (iii)}] $\beta(T_{b,d})=2n+3b+3d$,
  \item[{\em (iv)}] $\gamma_{t,coi}(T_{b,d})=n+2b+2d$.
\end{itemize}
\end{remark}

According to the results above, for any positive integers $b,d$ we see that the tree $T_{b,d}\in \mathcal{F}$ satisfies
$$\gamma_{t,coi}(T)-(|V(T)|-\beta(T))=b\;\;\mbox{ and }\;\;(|V(T)|-|L(T)|)-\gamma_{t,coi}(T)=d,$$
which gives answer to our previous question.

\section*{Acknowledgement}

We want to thank the reviewer of this article for the useful comments that helped us to improve our work.

\end{document}